\documentclass[12pt]{article}
\usepackage{amsmath, amssymb, braket, amsthm, graphicx, titlefoot, color}
\usepackage{setspace}
\usepackage[square,numbers]{natbib}
\newtheorem{definition}{Definition}
\newtheorem{theorem}{Theorem}

\newtheorem{proposition}{Proposition}

\title{Interplay Between Game Theory and Control Theory}
\author{Souma Mazumdar  \\ Department of Theoretical Sciences \\ S. N. Bose National Centre for Basic Sciences \\ Block - JD, Sector - III, Salt Lake City, Kolkata - 700 106 \\ Email: souma.mazumdar@bose.res.in,  Phone: 09903144810}
\date{}

\begin{document}
	\maketitle 
	
	\begin{abstract}
        An interrelationship between Game Theory and Control Theory is seeked. In this respect two aspects of this relationship are brought up. To establish the direct relationship Control Based Games and to establish the inverse relationship Game Based Control are discussed. In the attempt to establish the direct relationship Control Based Boolean Networks are discussed with the novel technique of application of Semi Tensor Product of matrices in Differential Calculus. For the inverse relationship $H^{\infty}$ robust optimal control has been discussed with the help of Dynamic Programming and Pontryagin Minimization Principle. \\ \\
        \noindent{\bf Keywords: Dynamic Game, Boolean Network, Semi Tensor Product, $H^{\infty}$ control, Riccati Equation, Pontryagin Minimization Principle} \\ \\
	\end{abstract}

	\section{Introduction}
    Game Theory as developed by Von Neumann and John Nash during the early 30's saw its first incarnation in static form without any consideration of involvement of temporal parameters. This was what the game between two players was represented in static bimatrix or normal form. But with the passage of decades the need for introduction of time parameter was felt in game theory which led to the development of dynamic game theory which can be pictorially represented in extensive form or a game tree evolving with the passage of time. With the flow of evolutionary ideas in game theory and the birth of Evolutionary Game Theory the need for a time parameter in the existing theory was felt more acute. Also the number of players were thought to be couldnot be limited in two but taking into account involvement of an arbitrary number of players in a particular game. As time became an indispensable component in the existing theory, there were ideas to capture the time evolution of the game in terms of differential(for continous time) or difference(for discrete time) equations. And in this respect the most perfect idea was to borrow theories from Control Theory which was developed on the pillars of differential equations of time. And it brought the merging of two theories the Game Theory and the Control Theory. The last two decades saw an extensive application of Mathematical Control Theory in Game Theory and now it is thought that they both complement each other for the develpoment of a more robust theory where one of theories is assumed to be incomplete without the other.\par
    There is a two way road between Game Theory and Control Theory where juggling of ideas continue between the two theories. On one hand there is application of concepts of Control Theory in Game Theory which is known as Control based Games and on the other hand application of concepts of Game Theory to approach problems in Control Theory which is known as Game based Control. In Evolutionary Games the whole game is pictorially represented as a network graph where each node of the graph represents a species whose time evolution should be studied. In more simplified language it is the time evolution of each node of the graph which changes its state with the passage of time. In some seminal works\cite{cheng2014semi,cheng2015modeling,cheng2010analysis,cheng2010linear,li2011structure,cheng2009input,cheng2011stability,zhang2012multi} Cheng and his collaborators showed that the network describing the game can be thought of a boolean network where the nodes changes its state between 0 and 1 as a response after each interaction with its neighboring nodes. In this develpoment it was primarily considered that a node can respond and change its state only with the nodes with which it is edge connected and not between distant nodes. As a more complicated scenario $k$ valued networks\cite{li2010algebraic} can also exist where each node can take $k$ different values instead of only two values that is 0 and 1. But in our work we restrict ourself to only boolean networks. In this respect we seek to apply the mathematical ideas of semi tensor product to analyse boolean networks as introduced by Cheng and his collaborators\cite{cheng2005semi,qi2014networked,cheng2014networked,cheng2003semi}. This is a picture of how Control Theory was involved to tackle the problems of modern Game Theory\cite{cheng2010strategy}. As already mentioned it is a two way road between the two theories there is another aspect of it. That is how Game Theory got its application to approach the problems of modern Control Theory. Game Theory finds a direct application in $H^{\infty}$ robust control\cite{limebeer1992game,bernhard1991lecture,basar1991dynamic,bacsar1989disturbance,bacsar1989differential,haurie2005dynamic,bacsar1984theory}. $H^{\infty}$ robust control\cite{francis1987course} is concerned with design of suitable control parameters which is most suited for a particular plant to withstand the adverse effects of natural hazards or measurement noise which may interfere with the operation of the plant. Here two players are considered, the controller and the nature. And thus there are two inputs one in the form of control and another in the form of cumulative adverse effects which the plant has to withstand. It borrows the formalism of state variable approach from Control Theory\cite{doyle1989state}. The aim of the controller is to design a perfect control to maximize the performance index of the plant in worst case scenario of the adverse effects which nature may inflict on the plant. Thus it is a competitive game between the controller and nature where controller wishes to maximize the performance index and nature tries to minimize it through its adverse effects. Thus it reduces to a \textit{min-max} optimization problem\cite{bacsar2008h}. In $two-person-zero-sum$ games also we are confronted with such a \textit{min-max} optimization problem. The optimum point is the saddle point of the optimization which is also the Nash Equilibrium in terms of Game Theoretic concepts. So it seems natural to use the ideas and tools of $two-person-zero-sum$ game in $H^{\infty}$ robust control. \par
    In our quest to establish a interrelationship between the two apparently disconnected theories we take up Control Based Games in the first part of our discussion. There we introduce an absolutely novel technique in implementing the Semi-Tensor product formalism in Markov Decision Process(MDP) dynamics\cite{vrabie2013optimal,zhao2010optimal,pan1995h}. There we derive the Bellman Optimality Equation for Boolean networks where we successfully applied the properties of Semi-Tensor product while doing the differential calculus. Finally we arrive at the computation of optimal control by considering a value of state trajectory. In the second part of the discussion we have dealt with $H^{\infty}$ optimal control for discrete time systems\cite{abouheaf2014multi}. By considering the dynamic programming for discrete time systems we state the Pontryagin Minimization Principle \cite{basar1999dynamic} for the same and also present a thought problem in the form of a proposition. While computing the optimal control, we adopted an absolutely new technique of doing the delta$(\Delta)$ variations of the state equation and eventually replacing the delta variations by differentials which made the equations much simpler. The matrix Riccati equations\cite{bernstein1989lqg,bacsar2017riccati} which came up in this process were solved successfully for some numerical values of the matrices. Finally we calculate the spectral radius and through the condition on the \textit{attenuation level} $\gamma$ we attempt to calculate a range for $\gamma$. 
    
    \section{Basic Idea and our Approach}  While dealing with the first part of our discussion that is Control Based Games we consider the representation of the game as a network graph. We assume it is a Boolean network and apply the formalism of semi tensor product(STP) as introduced by Cheng and his collaborators. STP approach which is generalisation of conventional matrix product uses STP to express a logical equation into matrix form which makes it possible to convert a logical dynamic system into a discrete time system. In the next section while introducing the basic notation we do a brief review of STP formalism. In Network Evolutionary Games(NEG) a key issue is the strategy updating rule. That is, how a player chooses his strategy based on his information about his neighborhood players. We try to capture the time evolution of the strategies of individual nodes in discrete time by equating its state in $t+1$ th interval depending on its state at $t$ th interval with a matrix equation.\par
    While dealing with the second part of our discussion that is Game Based Control our prime motivation is to show the application of \textit{two-person-zero-sum-game} in the problem of $H^{\infty}$ optimal control. The main idea goes like the following: \\
    Given a plant $y=Gu$ devise a feedback control $u=Ky$ which shall make the \textit{sensitivity transfer function} $T=(I+GK)^{-1}$ small in order to be robust. Here "small" refers to the maximum value of the norm of $T(iw)$ over all frequencies, i.e. the $H^{\infty}$ norm of $T$ has to lie in the Hardy Space $H^{\infty}$ in order that the closed loop system to be stable. Further, for measurement of noise insensitivity it is desired that the \textit{complementary sensitivity function} $T_{c}=GK(I+GK)^{-1}$ also be small. However, since $T+T_{c}=1$, it is impossible to make both small at the same time. The solution to this problem is to realize that modelling errors introduce \textit{low frequency} disturbances while measurement errors tend to be \textit{high frequency}. Thus one attempts to control the magnitude of $T$ at low frequencies and that of $T_{c}$ at high frequencies. When the standard solution to this problem was expressed in terms of Riccati Equations it was realised that the problem could be stated in terms of $min-max$ linear quadratic problem which forms the basic foundation of the problem formulation and its solution, as we will see in our subsequent discussions.\par
    The rest of the paper is organized as follows: \\
    We broadly divide the paper into two sections, one dedicated for the direct and the other, the inverse relationship\cite{isaacs1999differential,basar1985dynamic,basar2018handbook} that exists between Game Theory and Control Theory. Then in individual sections we introduce the basic and necessary mathematical structure that we are going to follow along with the problem describition and its solution. Finally we conclude with some open problems.
    
    \begin{center}
    	 \textbf{I. CONTROL BASED GAMES}
    \end{center}
    \section{Formulation of Networked Evolutionary Games}
    \subsection{Important Notations and basic formalism}
    1. $M_{m \times n}$ is the set of $m \times n$ real matrices. \\
    2. Col$_{i}(M)$ is the $i$-th column of matrix $M$;Col$(M)$ is the set of columns of $M$. \\
    3. $\mathcal{D}_{k}:={1,2, \dots, k}$ \\
    4. $\delta^{i}_{n}:=$Col$_{i}(I_{n})$ $i.e.$ it is the $i$-th column of the identity matrix. \\
    5. $\Delta:=$Col$(I_{n})$ \\
    6. $M \in \mathcal{M}_{m \times n}$ is called a logical matrix if Col$(M) \subset \Delta_{m}$ the set of $m \times n$ logical functions is denoted by $\mathcal{L}_{m \times n}$ \\
    7. Assume $L \in \mathcal{L}_{m \times n}$, then
    $$L=[\delta^{i_{1}}_{m} \;  \delta^{i_{2}}_{m} \dots \delta^{i_{n}}_{m}]$$;
    and its shorthand form is 
    $$L=\delta_{m}[i_{1} \; i_{2} \dots i_{n}]$$
    8. A $k$ dimensional vector with all entries equal to $1$ is denoted by
    $$1_{k}:=(1 \; \dots 1)^{T}$$
    9. $A \ltimes B$ is the semi-tensor product(STP) of two matrices $A$ and $B$. The symbol $\ltimes$ is mostly omitted and we express $$AB:= A \ltimes B$$
    Here we take the opportunity to briefly introduce the definition of STP. 
    \begin{definition}
    	 Let $A \in \mathcal{M}_{m \times n}$ and $B \in \mathcal{M}_{p \times q}$. Denote by $t:=$lcm$(n,p)$. Then we define the semi-tensor product(STP) of $A$ and $B$ as 
    \begin{equation}
        A \ltimes B:=(A \otimes I_{\frac{t}{n}})(B \otimes I_{\frac{t}{p}}) \in \mathcal{M}_{(\frac{mt}{n}) \times (\frac{qt}{p})} 
    \end{equation}
    \end{definition}
    It is to be noted when $n=p$, $A \ltimes B=AB$. So the STP is a generalisation of conventional matrix product. STP keeps almost all the major properties of the conventional matrix product unchanged. \\
    We discuss some basic properties of STP. \\
    $a.$ Associative Law :
    \begin{equation}
    A \ltimes (B \ltimes C)=(A \ltimes B)\ltimes C 
    \end{equation} 
    $b.$ Distributive Law :
    \begin{equation}
    \begin{split}
    & (A+B) \ltimes C = A \ltimes C + B \ltimes C \\
    & A \ltimes(B+C)= A \ltimes B + A \ltimes C 
    \end{split} 
    \end{equation}
    $c.$ Transpose :
    \begin{equation}
    (A \ltimes B)^{T}=B^{T} \ltimes A^{T} 
    \end{equation}
    $d.$ Inverse:\\
    If $A$ and $B$ are invertible then
    \begin{equation}
    (A \ltimes B)^{-1}=B^{-1} \ltimes A^{-1} 
    \end{equation}
    $e.$ Let $X \in \mathbb{R}^{t}$ be a column vector. Then for matrix $M$
    \begin{equation}
    X \ltimes M =(I_{t} \otimes M)\ltimes X 
    \end{equation}	
    10. Let $f: \mathcal{B}^{n} \rightarrow \mathcal{B}$ be a boolean function expressed as
    \begin{equation}\label{relation}
    y=f(x_{1}, \dots, x_{n}) 
    \end{equation}
    where $\mathcal{B}=\{0,1\}$. Identifying
    \begin{equation}\label{booldef}
    \begin{split}
    1 = \delta^{1}_{2} = [1 \; 0]^{T}, \; 0 = \delta^{2}_{2}= [0 \; 1]^{T} 
    \end{split} 
    \end{equation}
    Then there exists a unique logical matrix $M_{f} \in \mathcal{L}_{2 \times 2^{n}}$ called the structure matrix of $f$ such that under vector form by using (\ref{booldef}), (\ref{relation}) can be expressed as 
    \begin{equation}
    y=M_{f} \ltimes_{i=1}^{n} x_{i} 
    \end{equation}
    which is called the algebraic form of (\ref{relation}).
    \section{Markov Decision Process in Boolean Networks}
    We introduce an absolutely novel approach to implement the ideas of Markov Decision Process(MDP) which is a variant of Reinforcement Learning as a new technique in Boolean Networks.
    \subsection{Markov Decision Process(MDP) dynamics for discrete time systems}
    Consider the discrete time Linear Quadratic Regulator(LQR) problem where MDP satisfies the state transition equation
    \begin{equation}\label{state_update}
    \begin{split}
    x_{k+1}=Ax_{k}+Bu_{k} 
    \end{split} 
    \end{equation}
    where $x_{k} \in X=\mathbb{R}^n$, the state space and $u_{k} \in U=\mathbb{R}^m$, the control space. $A$ and $B$ are matrices of appropriate dimensions and $k$ being the discrete time index. \\
    The infinite horizon performance index is given by
    \begin{equation}
    \begin{split}
    J_{k}=\frac{1}{2}\sum_{i=k}^{\infty}(x_{i}^{T}Qx_{i}-u_{i}^{T}Ru_{i})
    \end{split}  
    \end{equation}
    where the cost weighting matrices satisfy $Q=Q^{T} \geq 0$ and $R=R^{T} > 0$.\\ \\
    \textbf{a. Bellman equation for discrete time LQR, the Lyapunov Equation:} \\
    The controls are assumed to be the policies adopted by the players in the network. For a policy  $u_{k}=\mu_{k}$ at $k$ th time the associated value function is given by 
    \begin{equation}
    \begin{split}
    V(x_{k})& =\frac{1}{2}\sum_{i=k}^{\infty}(x_{i}^{T}Qx_{i}-u_{i}^{T}Ru_{i}) \\ &
    =\frac{1}{2}(x_{k}^{T}Qx_{k}-u_{k}^{T}Ru_{k})+\frac{1}{2}\sum_{i=k+1}^{\infty}(x_{i}^{T}Qx_{i}-u_{i}^{T}Ru_{i}) \\ &
    =\frac{1}{2}(x_{k}^{T}Qx_{k}-u_{k}^{T}Ru_{k})+V(x_{k+1}) 
    \end{split} 
    \end{equation}
    Assuming the cost is quadratic so that $V(x_{k})=\frac{1}{2}x_{k}^{T}Px_{k}$ for some kernel matrix $P > 0$ yields the Bellman equation
    \begin{equation}\label{P_symmetric}
    \begin{split}
    2V(x_{k})=x_{k}^{T}Px_{k}=x_{k}^{T}Qx_{k}-u_{k}^{T}Ru_{k}+x_{k+1}^{T}Px_{k+1} 
    \end{split} 
    \end{equation}
    \begin{proposition}
    	The matrix $P$ in equation (\ref{P_symmetric}) is symmetric that is $P=P^{T}$
    \end{proposition}
    \begin{proof}
    	From equation (\ref{P_symmetric}) we have
    	\begin{equation}\label{P_symmetric1}
    	\begin{split}
    	x_{k}^{T}Px_{k}=x_{k}^{T}Qx_{k}-u_{k}^{T}Ru_{k}+x_{k+1}^{T}Px_{k+1} 
    	\end{split} 
    	\end{equation}
    	Taking the transpose of equation (\ref{P_symmetric1}) we have
    	\begin{equation}\label{P_symmetric_transpose}
    	\begin{split}
    	x_{k}^{T}P^{T}x_{k} & =x_{k}^{T}Q^{T}x_{k}-u_{k}^{T}R^{T}u_{k}+x_{k+1}^{T}P^{T}x_{k+1} \\ &
    	=x_{k}^{T}Qx_{k}-u_{k}^{T}Ru_{k}+x_{k+1}^{T}P^{T}x_{k+1} \; \; [\because{Q=Q^{T},R=R^{T}}] 
    	\end{split} 
    	\end{equation}
    	Subtracting equation (\ref{P_symmetric1}) from (\ref{P_symmetric_transpose}) we have
    	\begin{equation}
    	\begin{split}
    	x_{k}^{T}(P-P^{T})x_{k}=x_{k+1}^{T}(P-P^{T})x_{k+1} 
    	\end{split} 
    	\end{equation}
    	Now as this equation is true for any $x_{k}$ and $x_{k+1}$ we should have
    	\begin{equation}
    	\begin{split}
    	& P-P^{T}=0 \\ &
    	\implies P=P^{T} 
    	\end{split} 
    	\end{equation}
    	i.e. $P$ is symmetric.
    \end{proof}
    \textbf{b. Bellman optimality equation for discrete time LQR:} \\
    The discrete time LQR Hamiltonian function is given by
    \begin{equation}
    \begin{split}
    H(x_{k},u_{k})&=x_{k}^{T}Qx_{k}-u_{k}^{T}Ru_{k}+x_{k+1}^{T}Px_{k+1}-x_{k}^{T}Px_{k} \\ &
    =x_{k}^{T}Qx_{k}-u_{k}^{T}Ru_{k}+(Ax_{k}+Bu_{k})^{T}P(Ax_{k}+Bu_{k})-x_{k}^{T}Px_{k} \\ & \text{[Replacing $x_{k+1}$ from equation (\ref{state_update})]} 
    \end{split} 
    \end{equation}
    If $u_{k}^{\star}$ stands for optimal control with $(\star)$ signifying the condition for optimality we have
    \begin{equation}
    \begin{split}
    \frac{\partial H}{\partial u_{k}^{\star}}=0 
    \end{split} 
    \end{equation}
    Solving the above equation gives the optimal control as
    \begin{equation}\label{optimal_control}
    \begin{split}
    u_{k}^{\star}=-(B^{T}PB-R)^{-1}B^{T}PAx_{k}^{\star}
    \end{split} 
    \end{equation}
    where $x_{k}^{\star}$ stands for the trajectory in optimal condition.
    \subsection{Markov Decision Process for Boolean Networks}
    We try to adopt the formalism of MDP for discrete time system for Boolean Control Networks. Let the system consists of $n$ states and $m$ controls which are represented by the nodes of the Boolean network. We write $x_{k}=\ltimes_{i=1}^{n}x_{k}^{i}$ and $u_{k}=\ltimes_{j=1}^{m}u_{k}^{j}$. Then the state updation law is given by
    \begin{equation}\label{state_update_Boolean}
    \begin{split}
    x_{k+1}&=L \ltimes u_{k} \ltimes x_{k} \\ &
    =Lu_{k}x_{k} \; \; \text{[Where the notations for STP are omitted for convenience]} \\ &
    \text{where $L$ is known as the state transition matrix} 
    \end{split} 
    \end{equation}
    \textbf{a. Bellman optimality equation for Boolean Networks:} \\
    The discrete time LQR Hamiltonian function is given by
    \begin{equation}\label{Hamiltonian_Boolean}
    \begin{split}
    H(x_{k},u_{k})&=x_{k}^{T}Qx_{k}-u_{k}^{T}Ru_{k}+x_{k+1}^{T}Px_{k+1}-x_{k}^{T}Px_{k} \\ & 
    \end{split} 
    \end{equation}
    If $u_{k}^{\star}$ denotes the optimal control we should have
    \begin{equation}
    \begin{split}
    \frac{\partial H}{\partial u_{k}^{\star}}=0 
    \end{split} 
    \end{equation}
    Doing the partial derivative of equation (\ref{Hamiltonian_Boolean}) with respect to $u_{k}$ and setting it to $0$ we have
    \begin{equation}\label{Bellman_for_Boolean}
    \begin{split}
    & -2u_{k}^{\star T}R + 2x_{k+1}^{\star T}P\frac{\partial x_{k+1}^{\star}}{\partial u_{k}^{\star}} = 0 \\ &
    -2u_{k}^{\star T}R + 2x_{k}^{\star T}u_{k}^{\star T}L^{T}Px_{k}^{\star}L=0 
    \end{split} 
    \end{equation}
   Where the products are understood as STP. We replaced the value of $x_{k+1}$ and have taken its partial derivative with respect to $u_{k}$. $u_{k}^{\star}$ and $x_{k}^{\star}$ denotes the optimal control and optimal trajectory respectively. \\
    Simplifying equation(\ref{Bellman_for_Boolean}) we have
    \begin{equation}\label{Bellman_optimal_for_Boolean}
    \begin{split}
    & u_{k}^{\star T}R=x_{k}^{\star T}u_{k}^{\star T}L^{T}Px_{k}^{\star}L \\ &
    u_{k}^{\star T}=x_{k}^{\star T}u_{k}^{\star T}L^{T}Px_{k}^{\star}LR^{-1} 
    \end{split} 
    \end{equation}
    Now $P$ is a symmetric matrix and $R$ is also a symmetric matrix. Without loss of generality we can assume $P=I_{2^{n}}$(Identity matrix of dimension $2^{n} \times 2^{n}$) and $R^{-1}=R=I_{2^{m}}$(Identity matrix of dimension $2^{m} \times 2^{m}$). \\
    Then equation (\ref{Bellman_optimal_for_Boolean}) simplifies to
    \begin{equation}\label{optimal_control_Boolean}
    \begin{split}
    u_{k}^{\star T}=x_{k}^{\star T}u_{k}^{\star T}L^{T}x_{k}^{\star}L 
    \end{split} 
    \end{equation}
    Taking the transpose of equation (\ref{optimal_control_Boolean}) we have
    \begin{equation}\label{optimal_control_Boolean_transpose}
    \begin{split}
    u_{k}^{\star}=L^{T}x_{k}^{\star T}Lu_{k}^{\star}x_{k}^{\star} 
    \end{split} 
    \end{equation}
    We assume $x_{k}^{\star}=\delta_{2^{n}}^{i}$ and with no loss of generality $L=(\delta_{2^{m}}^{i})^{T}$. \\
    Then the above equation reduces to
    \begin{equation}\label{optimal_control_Boolean_values_replaced}
    \begin{split}
    u_{k}^{\star}=(\delta_{2^{m}}^{i})(\delta_{2^{n}}^{i})^{T}(\delta_{2^{m}}^{i})^{T}u_{k}^{\star}(\delta_{2^{n}}^{i}) 
    \end{split} 
    \end{equation}
    Where all the above products are understood to be STP.
    \begin{proposition}
    	There exists a unique $u_{k}^{\star}$ of matrix dimension $2^{m} \times 1$ which satisfies equation(\ref{optimal_control_Boolean_values_replaced}).
    \end{proposition}
    \begin{proof}
    	$u_{k}$ is the STP of $m$ controls $u_{k}=\ltimes_{j=1}^{m}u_{k}^{j}$. Thus it is a vector of dimension $2^{m} \times 1$. We try to compute the matrix dimension of resultant matrix of the STP at the R.H.S. of equation (\ref{optimal_control_Boolean_values_replaced}). Let us take two terms in the R.H.S of equation (\ref{optimal_control_Boolean_values_replaced}), find the STP and again do the STP with the third terms and so on as STP is associative.
    	Dimensions of:
    	$\delta_{2^{m}}^{i}=2^{m} \times 1$, 
    	$(\delta_{2^{n}}^{i})^{T} = 1 \times 2^{n}$, 
    	$(\delta_{2^{m}}^{i})^{T} = 1 \times 2^{m}$,
    	$u_{k}^{\star} = 2^{m} \times 1$,
    	$\delta_{2^{n}}^{i} = 2^{n} \times 1$. \\
    	Doing the STP of first two terms generates a matrix of dimension $2^{m} \times 2^{n}$. Taking STP of this matrix with the third term generates a matrix of dimension $2^{m} \times 2^{m+n}$. Taking the STP of this matrix with the fourth term generates a matrix of dimension $2^{m} \times 2^{n}$. Taking the STP of this matrix with the last term generates a matrix of dimension $2^{m} \times 1$. So we see the resultant matrix of the R.H.S of equation (\ref{optimal_control_Boolean_values_replaced}) is of dimension $2^{m} \times 1$. The L.H.S of the same equation is also a matrix of dimension $2^{m} \times 1$. As the matrix dimensions match in L.H.S and R.H.S we can say the unique matrix $u_{k}^{\star}$ exists.
    	\end{proof}
    	\textbf{a. Determination of $u_{k}^{\star}$:}
    	By deep inspection and a tedious calculation we arrive at the value of $u_{k}^{\star}$ as $u_{k}^{\star}=\delta_{2^{m}}^{i}$. \\ \\
       	\textbf{b. Determination of optimal control and optimal trajectory:} 
    	We have from equation (\ref{state_update_Boolean}), 
    	$$x_{k+1}=Lu_{k}x_{k}$$ \\   	
    	In optimal condition the above equation takes the form
    	\begin{equation*}
    	\begin{split}       	 x_{k+1}^{\star}=Lu_{k}^{\star}x_{k}^{\star} 
         \end{split}  
         \end{equation*} 
 Replacing the value of $k$ with $0,1, \dots$ and so on we have,
 \begin{equation*}
 \begin{split}  
 \\ &	x_{1}^{\star}=Lu_{0}^{\star}x_{0}^{\star} \\ &
  	x_{2}^{\star}=Lu_{1}^{\star}x_{1}^{\star} \\ &
  	x_{3}^{\star}=Lu_{2}^{\star}x_{2}^{\star} \\ & \vdots \\ & \text{and so on}
  	\end{split}
     \end{equation*}
       	Puting $u_{0}^{\star}=\delta_{2^{m}}^{i},x_{0}^{\star}=\delta_{2^{n}}^{i}, L=(\delta_{2^{m}}^{i})^{T}$ we evaluate $x_{1}^{\star}$ as  $x_{1}^{\star}=\delta_{2^{n}}^{i}$.
    	From equation (\ref{optimal_control_Boolean_transpose}), we can say $x_{1}^{\star}=\delta_{2^{n}}^{i}$ will generate $u_{1}^{\star}=\delta_{2^{m}}^{i}$. So we see in optimal condition the values are repeated for $x_{k}$ and $u_{k}$ for each discrete time indexed by $k$. So we can say that the system falls into a cycle of length $1$ where the values are repeated over each discrete time. Such a cycle is called a fixed point.
        \begin{center}
        	\textbf{II. GAME BASED CONTROL}
        \end{center}
        Now we turn our attention for an attempt to try to show how the inverse relationship exists between Game Theory and Control Theory where solutions to problems of Control Theory are dictated by the principles of Game Theory. We pick up a unique aspect of this relationship for our attempt to show this inverse relationship. Game theory has profound and successful application to tackle the problems of $H^{\infty}$ robust optimal control. The basic idea accompanying this attempt is given in the subsequent sections.
        \section{Problem Identification and probable solution}
        \subsection{Min-max problems in robust control}
        Let $U$ be a decision space and $W$ a disturbance space. Let $J :U \times V \rightarrow \mathbb{R}$ depending on decision $u \in U$ and on an unknown disturbance $w \in W$. The decision maker, choosing $u$ that is the control wants to make J as small as possible in spite of the a priori unknown disturbance $w \in W$. Then the guranteed performance of a given decision $u \in U$ is any number $g$ such that
        \begin{equation*}
        \begin{split}
        J(u,w) < g, \; \; \forall w \in W
        \end{split}
        \end{equation*}
        Clearly, the best guaranteed performance for a given decision $u$ is 
        \begin{equation*}
        \begin{split}
        G(u)=\sup_{w \in W} J(u,w)
        \end{split}
        \end{equation*}
        Now, the problem of finding the best possible decision in this context is to find
        the smallest guaranteed performance, or
        \begin{equation}
        \begin{split}
        \inf_{u \in U}G(u) = \inf_{u \in U} \sup_{w \in W} J(u,w) 
        \end{split} 
        \end{equation}
        If the infimum in $u$ is reached, then the minimizing decision $u^{\star}$ is called the optimal decision or the optimal control. \\
        Let $U$ and $V$ be normed vector spaces $Z$ be an auxiliary normed vector space $z \in Z$ the output whose norm is to be kept small in spite of the disturbances $w$. We assume that for each decision $u \in U$, $z$ depends lineraly on $w$. Therefore there is a nonlinear application $P: U \rightarrow \mathcal{L}(W \rightarrow Z)$ and
        \begin{equation*}
        z=P(u)w
        \end{equation*}
        Clearly $z$ cannot be kept bounded if $w$ is not. A natural formalization of the problem of keeping it small is to try and make the operator norm of $P(u)$ as small as possible. In terms of guaranteed performance $||P(u)|| \leq \gamma$ for a given postive \textit{attenuation level} $\gamma$. This is equivalent to
        \begin{equation}\label{inequality_1}
        \begin{split}
        & ||z|| \leq \gamma ||w||, \; \; w \in W \\ &
        \text{or equivalently} \; ||z||^{2} \leq \gamma^{2}||w||^{2}, \; \; w \in W \\ &
        \text{or equivaletly again,} \;
        \sup_{w \in W}[||P(u)w||^{2} - \gamma^{2} ||w||^{2}] \leq 0
        \end{split} 
        \end{equation}
        Now, given a number, $\gamma$, this has a solution that is there exists a decision $u \in U$ satisfying that inequality $if$ the infimum hereafter is reached or is negative, and \textit{only if} 
        \begin{equation}\label{inequality_2}
        \begin{split}
        \inf_{u \in U} \sup_{w \in W} [||P(u)w||^{2}-\gamma^{2}||w||^{2}] \leq 0
        \end{split}
        \end{equation}
        This is the method of $H^{\infty}$ optimal control.
        \section{$H^{\infty}$ optimal control}
        Given a linear system with inputs $u$ and $w$ and a desired attenuation level $\gamma$ we want to know whether there exist causal control laws satisfying inequality (\ref{inequality_1}) and if yes, find one. This is the \textit{standard problem} of $H^{\infty}$ optimal control. We propose here an approach of this problem based upon dynamic game theory. Since we have been dealing with discrete time systems from the beginning this discussion is also restricted to discrete time systems though it can be formulated equivalently in continous time.
        \subsection{Discrete Time}
        We consider now the discrete-time system where $t \in \mathbb{N}$:
        \begin{equation}
        x_{t+1}=A_{t}x_{t} + B_{t}u_{t} + D_{t}w_{t}, \; x_{t_{0}}=x_{0} 
        \end{equation}
        \begin{equation}
        y_{t}=C_{t}x_{t} + E_{t}w_{t} 
        \end{equation}
        \begin{equation}
        z_{t}=H_{t}x_{t}+G_{t}u_{t} 
        \end{equation}
        where the system matrices may depend on the time $t$. Then we want to optimize the performance index given by\cite{bernhard2016robust}
        \begin{equation}\label{Cost_function}
        \begin{split}
        J=||x_{t_{1}}||^{2}+\sum_{t=t_{0}}^{t_{1}-1}(||z_{t}||^{2}-\gamma^{2}||w_{t}||^{2})-\gamma^{2}||x_{0}||^{2}
        \end{split} 
        \end{equation}
        We introduce the following notations
        \begin{equation}
        \begin{split}
        \bar{A}_{t} = A_{t}-B_{t}R^{-1}_{t}P^{T}_{t} \; , \; \tilde{A}_{t}=A_{t}-L^{T}_{t}N^{-1}_{t}C_{t}
        \end{split} 
        \end{equation}
        where $R_{t}=G^{T}_{t}G_{t}$, $P_{t}=H_{t}^{T}G_{t}$, $L_{t}=E_{t}D_{t}^{T}$, $N_{t}=E_{t}E_{t}^{T}$. \\
        Along with the additional notations
        \begin{equation}
        \begin{split}
        & \Gamma_{t}=(S_{t+1}^{-1}+B_{t}R_{t}^{-1}B_{t}^{T}-\gamma^{-2}M_{t})^{-1} \\ &
        \bar{S}_{t}=\bar{A}^{T}_{t}(S_{t+1}-\gamma^{-2}M_{t})^{-1}\bar{A}_{t}+Q_{t}-P_{t}R_{t}^{-1}P_{t}^{T} \\ &
        \Delta_{t}=(\Sigma_{t}^{-1}+C_{t}^{T}N_{t}C_{t}-\gamma^{-2}Q_{t})^{-1} \\ &
        \tilde{\Sigma}_{t+1}=\tilde{A}_{t}(\Sigma_{t}^{-1}-\gamma^{-2}Q_{t})^{-1}\tilde{A}_{t}^{T}+M_{t}-L_{t}^{T}N_{t}^{-1}L_{t} \\ &
        \end{split} 
        \end{equation}
        Then the two matrix Riccati equations needed for our purpose can be written as 
        \begin{equation}\label{Riccati_S}
        \begin{split}
        S_{t}=\bar{A}_{t}^{T}\Gamma_{t}\bar{A}_{t}+Q_{t}-P_{t}R_{t}^{-1}P_{t}^{T}
        \end{split} 
        \end{equation}
        \begin{equation}\label{Riccati_Sigma}
        \begin{split}
        \Sigma_{t+1}=\tilde{A}_{t}\Delta_{t}\tilde{A}_{t}^{T}+M_{t}-L_{t}^{T}N^{-1}_{t}L_{t}
        \end{split} 
        \end{equation}
        where $Q_{t}=H_{t}^{T}H_{t}$, $M_{t}=D_{t}D_{t}^{T}$ \\
        \subsubsection{Computation of Optimal Control For Infinite Horizon Case}
        Here we deal with the special case of discrete time system along with infinite horizon that is the end time is infinity. This problem is known as infinite horizon stationary problem. The matrices are no longer time dependent and the dynamics is to be understood with zero initial condition at $-\infty$. The criterion (\ref{Cost_function}) is replaced by a sum from $-\infty$ to $\infty$ with no initial and terminal terms. \\ Throwing the problem in the formalism of Pontryagin Minimization Principle we have,
        \begin{equation}\label{Cost_Inf_Horizon}
        \begin{split}
        J&=\sum_{t=-\infty}^{\infty}(||z_{t}||^{2}-\gamma^{2}||w_{t}||^{2}) \\ &
        =\sum_{t=-\infty}^{\infty}(z_{t}^{T}z_{t}-\gamma^{2}w_{t}^{T}w_{t})
        \end{split} 
        \end{equation}
         The Hamiltonian is given by,
        \begin{equation}\label{Hamiltonian}
        \begin{split}
        H& =(z_{t}^{T}z_{t}-\gamma^{2}w_{t}^{T}w_{t})+p_{t}^{T}(Ax_{t}+Bu_{t}+Dw_{t}) \; \; \text{[Since matrices are time independent]} \\ &
        =x_{t}^{T}H^{T}Hx_{t}+x_{t}^{T}H^{T}Gu_{t}+u_{t}^{T}G^{T}Hx_{t}+u_{t}^{T}G^{T}Gu_{t}-\gamma^{2}w_{t}^{T}w_{t}+p_{t}^{T}(Ax_{t}+Bu_{t}+Dw_{t})
        \end{split}
        \end{equation}
        Minimizing the Hamiltonian with respect to the control $u$ we have,
        \begin{equation}\label{Minimize_H_{u}}
        \begin{split}
        0& =\frac{\partial H}{\partial u_{t}} \\ &
        =2x_{t}H^{T}H\frac{\partial x_{t}}{\partial u_{t}}+2u_{t}^{T}G^{T}H\frac{\partial x_{t}}{\partial u_{t}}+2x_{t}^{T}H^{T}G+2u_{t}^{T}G^{T}G+p_{t}^{T}A\frac{\partial x_{t}}{\partial u_{t}}+p_{t}^{T}B \\ &
        \end{split}
        \end{equation}
        We consider partials of the $x_{t}$ with respect to $u_{t}$ because state is dependent on control. \\
        To evaluate the partials we adopt the following strategy. \\
        We have,
        \begin{equation*}
        \begin{split}
        & x_{t+1}=Ax_{t}+Bu_{t}+Dw_{t} \\ &
        x_{t}=Ax_{t-1}+Bu_{t-1}+Dw_{t-1}
        \end{split}
        \end{equation*}
        Subtracting one from the other we have
        \begin{equation}
        \begin{split}
        & \Delta x_{t}=A\Delta x_{t}+B\Delta u_{t}+D\Delta w_{t} \\ &
        \implies [I-A]\Delta x_{t}=B\Delta u_{t}+D\Delta w_{t} \\ &
        \implies [I-A]\frac{\Delta x_{t}}{\Delta u_{t}}=B+D\frac{\Delta w_{t}}{\Delta u_{t}}
        \end{split}
        \end{equation}
        Replacing the $\Delta$ variations with $\partial$ that is differentials we have,
        \begin{equation}
        \begin{split}
        [I-A]\frac{\partial x_{t}}{\partial u_{t}}=B+D\frac{\partial w_{t}}{\partial u_{t}}
        \end{split}
        \end{equation}
        Now the controls $u_{t}$ and $w_{t}$ are independent. Therefore $\frac{\partial w_{t}}{\partial u_{t}}=0$ and we get the above equation as,
        \begin{equation}
        \begin{split}
        \frac{\partial x_{t}}{\partial u_{t}}=[I-A]^{-1}B
        \end{split}
        \end{equation}
        Puting this in equation (\ref{Minimize_H_{u}}) and simplifying we have,
        \begin{equation}
        \begin{split}
        2x_{t}^{T}H^{T}H[I-A]^{-1}B+2x_{t}^{T}H^{T}G+2u_{t}^{T}G^{T}G+2u_{t}^{T}G^{T}H[I-A]^{-1}B+p_{t}^{T}\left[A[I-A]^{-1}+I\right]B=0
        \end{split}
        \end{equation}
        Using Woodbury matrix identity we have the above equation as
        \begin{equation}\label{Woodbury}
        \begin{split}
        2x_{t}^{T}H^{T}H[I-A]^{-1}B+2x_{t}^{T}H^{T}G+2u_{t}^{T}G^{T}G+2u_{t}^{T}G^{T}H[I-A]^{-1}B+p_{t}^{T}[I-A]^{-1}B=0
        \end{split}
        \end{equation}
         The adjoint equation is given by,
        \begin{equation}\label{adjoint}
        \begin{split}
        -p_{t+1}^{T}&=\frac{\partial H}{\partial x_{t}}  \\ &
        =2x_{t}^{T}H^{T}H+2u_{t}^{T}G^{T}H+p_{t}^{T}A      
        \end{split}
        \end{equation}
        Now from transversality condition $p_{t+1}=0$.\\
        Therefore,
        \begin{equation}
        \begin{split}
        & 2x_{t}^{T}H^{T}H+2u_{t}^{T}G^{T}H+p_{t}^{T}A=0 \\ &
        \implies p_{t}^{T}=-2(x_{t}^{T}H^{T}H+u_{t}^{T}G^{T}{H})A^{-1}
        \end{split}
        \end{equation}
        Puting the value of $p_{t}^{T}$ in equation (\ref{Woodbury}) and simplying we have,
        \begin{equation}
        \begin{split}
        & x_{t}^{T}H^{T}H[I-A]^{-1}B+x_{t}^{T}H^{T}G+u_{t}^{T}G^{T}G+u_{t}^{T}G^{T}H[I-A]^{-1}B \\ & -(x_{t}^{T}H^{T}H+u_{t}^{T}G^{T}H)A^{-1}[I-A]^{-1}B=0
        \end{split}
        \end{equation}
        Simplifying the above equation we have,
        \begin{equation}
        \begin{split}
        u_{t}^{T}\left[G^{T}G+G^{T}H[I-A^{-1}][I-A]^{-1}B\right]=-x_{t}^{T}\left[H^{T}H[I-A^{-1}][I-A]^{-1}B+H^{T}G\right]
        \end{split}
        \end{equation}
        Taking transpose of the above equation and solving for $u$ we have,
        \begin{equation}
        \begin{split}
        u_{t}=-\left[G^{T}G+B^{T}[I-A^{T}]^{-1}[I-{(A^{T})}^{-1}H^{T}G]\right]^{-1}\left[B^{T}[I-A^{T}]^{-1}[I-{(A^{T})}^{-1}H^{T}H]\right]x_{t}
        \end{split}
        \end{equation}
        The above equation points towards the condition of optimality. So it can be written more meaningfully as,
        \begin{equation}\label{optimal_control}
        \begin{split}
        u_{t}^{\star}=-\left[G^{T}G+B^{T}[I-A^{T}]^{-1}[I-{(A^{T})}^{-1}H^{T}G]\right]^{-1}\left[B^{T}[I-A^{T}]^{-1}[I-{(A^{T})}^{-1}H^{T}H]\right]x_{t}^{\star}
        \end{split}
        \end{equation}
         \\ \\
        Here we take a little digression to steer the path of discussion in a new direction discussing about Dynamic Programming which is based on principles of optimality and the concept of Nash Equilibrium at the Optimal point. We also propose a thought problem very much relevant to this context.
        \section{Dynamic Programming for Discrete Time Systems}
        The method of dynamic programming is based on \textit{Principle of Optimality}.	which states that an optimal strategy has the property that, whatever the initial state and time are, all remaining decisions (from that particular initial state and particular initial time onwards) must also constitute an optimal strategy. To exploit this principle, we work backwards in time, starting at all possible final states with the corresponding final times. We now discuss the principle of optimality within the context of discrete-time systems but with only one player ($n=1$), but it can be generalised for ($n=N$) players. The discrete time state updation law and the cost function is given by
        \begin{equation}
        \begin{split}
        & x_{t+1}=f_{t}(x_{t},u_{t}), \; \; u_{t} \in U(\text{Control Space}) 
        \\ & 
        J(u)=\sum_{t=1}^{t=T}g_{t}(x_{t+1},u_{t},x_{t})
        \end{split} 
        \end{equation}
        In order to determine the minimizing control strategy, we shall need the expression for the \text{minimum cost} from any starting point at any initial time. This is also called the \textit{value function} and is defined as
        \begin{equation*}
        V(t,x)=\min_{u_{t} \in U} \left[\sum_{t=t_{0}}^{T}g_{t}(x_{t+1},u_{t},x_{t})\right]
        \end{equation*}
        A direct application of the principle of optimality now readily leads to the recursive relation
        \begin{equation}
        V(t,x)=\min_{u_{t} \in U}[g_{t}(f_{t}(x,u_{t}),u_{t},x)+V(t+1,f_{t}(x,u_{t})] 
        \end{equation}
        It is clear that $V(1,x)=J(u^{\star})$ where $u^{\star}$ is the optimal control. The point where this condition for optimality occurs is also the Nash Equilibrium of the game. Now we generalise the number of players to $N$ to state the following theorem. As it is a standard optimal control theorem it is stated without proof to refrain from delving into lengthy standard calculations.
        \begin{theorem}
        	For an $N$ person discrete time infinite dynamic game let, \\
        	$(i) \; f_{t}(u_{t}^{1} \dots u_{t}^{N})$ be continously differentiable on $\mathbb{R}^{n}$ \\
        	$(ii) \; g_{t}^{i}(u_{t}^{i} \dots u_{t}^{N})$ be continously differentiable on $\mathbb{R}^{n} \times \mathbb{R}^{n}$, $i \in n$. \\
        	Then if $u_{t}^{\star i}$ provides a Nash Equilibrium solution and $x_{t}^{\star}$ the corresponding state trajectory, there exists a finite sequence of $n$ dimensional costate vectors $(p_{2}^{i} \dots p_{t+1}^{i})$ for each $i \in n$ such that the following relations are satisfied.\\
        	\begin{equation}
        	\begin{split}
        	& x^{\star}_{t+1}=f_{t}(x_{t}^{\star},u_{t}^{1} \dots u_{t}^{N}), \; x_{1}^{\star}=x_{1} \\  &
        	u^{i \star}_{t}=\arg\min_{u_{t}}H_{t}^{i}(p_{t+1}^{i},u_{t}^{1 \star} \dots u_{t}^{N \star},x_{t}^{\star}) \\ &
        	p^{i}_{t}=\frac{\partial}{\partial x_{t}}f_{t}^{T}\left(x_{t}^{\star},u^{1 \star}_{t} \dots u^{N \star}_{t}\right)\left[p^{i}_{t+1}+\left(\frac{\partial}{\partial x_{t+1}}g_{t}^{i}\left(x^{\star}_{t+1},u^{1 \star}_{t} \dots u^{N \star}_{t},x_{t}^{\star}\right)\right)^{T}\right] \\ & +\left[\frac{\partial}{\partial x_{t}}g_{t}^{i}\left(x_{t+1}^{i},u_{t}^{1} \dots u_{t}^{N},x_{t}^{\star}\right)\right]^{T} \\ &
        	p^{i}_{T+1}=0, \; i \in n
          	\end{split} 
        	\end{equation}
        	where 
        	\begin{equation*}
        	H_{t}^{i}\left(p_{t+1},u^{1}_{t} \dots u^{N}_{t}.x_{t}\right)=g^{i}_{t}\left(f_{t}\left(x_{t},u^{1}_{t} \dots u_{t}^{N}\right),u^{1}_{t} \dots u^{N}_{t},x_{t}\right)+p^{i T}_{t+1}f_{t}\left(x_{t},u^{1}_{t} \dots u^{N}_{t}\right); \; \; i \in n
        	\end{equation*}
        	This set of equations is called Pontryagin Minimization Principle.
        \end{theorem}
       At this point we state the thought problem and also attempt to give its solution. As mentioned previously the point of optimality is also the Nash Equilibrium, we give the problem describition as a proposition below.
       \begin{proposition}
       	Consider the discrete time optimal control problem described by the following equations
       	\begin{equation}\label{discrete_time_system}
       	\begin{split}
       	& x_{t+1}=A_{t}x_{t}+B^{i}_{t}u^{i}_{t}+\sum_{j \in n}B^{j}_{t}u^{j \star}_{t} \\ &
       	J(u^{i})=\frac{1}{2}\sum_{t=t}^{T}(x^{T}_{t+1}Q^{i}_{t+1}x_{t+1}+u_{t}^{i T}R_{t}^{i}u_{t}^{i}); \;R^{i}_{t}=R^{i T}_{t} > 0, \; Q^{i}_{t+1}=Q^{i T}_{t+1} \geq 0
       	\end{split} 
       	\end{equation}
       	Then the Nash Equilibrium or equivalently the condition for optimality is given by solution of the set of relations
       	\begin{equation}
       	\begin{split}
       	& u^{i \star}_{t}=-P^{i}_{t}S^{i}_{t+1}A_{t}x_{t}^{\star}-P^{i}_{t}(s^{i}_{t+1}+S^{i}_{t+1}\sum_{i \in n}B^{i}_{t}u^{i \star}_{t}) \\ &
       	P^{i}_{t}=[R^{ii}_{t}+B^{i T}_{t}S^{i}_{t+1}B^{i}_{t}]^{-1}B^{i T}_{t} \\ &
       	S^{i}_{t}=Q^{i}_{t}+A^{T}_{t}S^{i}_{t+1}[I-B^{i}_{t}P^{i}_{t}S^{i}_{t+1}]A_{t}; \; S^{i}_{t+1}=Q^{i}_{t+1} \\ &
       	s^{i}_{t}=A^{T}_{t}[I-B^{i}_{t}P^{i}_{t}S^{i}_{t+1}]^{T}[s^{i}_{t+1}+S^{i}_{t+1}\sum_{i \in n} B^{i}_{t}u^{i \star}_{t}]; \; s^{i}_{t+1}=0; \; i \in n
        \\ &	\text{Furthermore the mimimum cost function is given by}
       	\\ &	 J(u^{i \star}) = \sum_{t=1}^{T}x^{T \star}_{t}A^{T}_{t}Q^{i}_{t+1}\left(-B^{i}_{t}P^{i}_{t}Q^{i}_{t+1}A_{t}x^{\star}_{t}-B^{i}_{t}P^{i}_{t}Q^{i}_{t+1}\left(\sum_{j \in n} B^{j}_{t}u^{j \star}_{t}\right) + \left(\sum_{j \in n} B^{j}_{t}u^{j \star}_{t}\right)\right) \\ &
       	+\frac{1}{2}\sum_{t=1}^{T}\left(x^{T}_{t}A^{T}_{t}+\left(\sum_{j \in n} B^{j}_{t}u^{j \star}_{t}\right)^{T}\right)Q^{i}_{t+1}P^{i T}_{t}\left(B^{i T}_{t}Q^{i}_{t+1}B^{i}_{t}+R^{i}_{t}\right)P^{i}_{t}Q^{i}_{t+1}\left(A_{t}x^{\star}_{t}+\left(\sum_{j \in n} B^{j}_{t}u^{j \star}_{t}\right)\right) \\ & -\sum_{t=1}^{T}\left(x^{T}_{t}A^{T}_{t}+\left(\sum_{j \in n} B^{j}_{t}u^{j \star}_{t}\right)^{T}\right)Q^{i}_{t+1}P^{i T}_{t}B^{i T}_{t}Q^{i}_{t+1}\left(\sum_{j \in n} B^{j}_{t}u^{j \star}_{t}\right)   \\ & + \frac{1}{2}\sum_{t=1}^{T}\left(\sum_{j \in n} B^{j}_{t}u^{j \star}_{t}\right)^{T}Q^{i}_{t+1}\left(\sum_{j \in n} B^{j}_{t}u^{j \star}_{t}\right) + \frac{1}{2}\sum_{t=1}^{T}\left[x^{T \star}A^{T}_{t}Q^{i}_{t+1}A_{t}x^{\star}_{t}\right] 
       	\end{split} 
       	\end{equation} 
       	
       	There is a slight abuse of notation by denoting the end time and the transpose of a matrix  
        by the same letter 'T'
       \end{proposition}
       \begin{proof}
       	From the \textit{Minimization Theorem} we have the associated Hamiltonian function as
       	\begin{equation}
       	\begin{split}
       	H^{i}_{t}& =\frac{1}{2}\sum_{t=1}^{T}[x^{T}_{t+1}Q^{i}_{t+1}x_{t+1}+u^{i T}_{t}R^{i}_{t}u^{i}_{t}]+p_{t+1}^{i T}\left(A_{t}x_{t}+B^{i}_{t}u_{t}^{i}+\left(\sum_{j \in n} B^{j}_{t}u^{j \star}_{t}\right)\right) \\ &
       	=\frac{1}{2}\sum_{t=1}^{T}\left[\left(A_{t}x_{t}+B^{i}_{t}u^{i}_{t}+ \sum_{j \in n}B^{j}_{t}u^{j}_{t}\right)^{T}Q^{i}_{t+1}\left(A_{t}x_{t}+B^{i}_{t}u^{i}_{t}+ \sum_{j \in n}B^{j}_{t}u^{j}_{t}\right)+u^{i T}_{t}R^{i}_{t}u^{i}_{t}\right] \\ & +p_{t+1}^{i T}\left(A_{t}x_{t}+B^{i}_{t}u_{t}^{i}+\left(\sum_{j \in n} B^{j}_{t}u^{j \star}_{t}\right)\right)
       	\end{split} 
       	\end{equation}
       	\begin{equation}
       	\begin{split}
       	-p_{t}^{i T}& =\frac{\partial H_{t}^{i}}{\partial x_{t}}=\left(A_{t}x_{t}\right)^{T}Q^{i}_{t+1}A_{t}+\left(B^{i}_{t}u^{i}_{t}+\left(\sum_{j \in n} B^{j}_{t}u^{j \star}_{t}\right)\right)^{T}Q^{i}_{t+1}A_{t}+p_{t+1}^{i T}A_{t} \\ &  =x^{T}_{t}A^{T}_{t}Q^{i}_{t+1}A_{t}+\left(u^{i T}_{t}B^{i T}_{t}+\left(\sum_{j \in n} B^{j}_{t}u^{j \star}_{t}\right)^{T}\right)Q^{i}_{t+1}A_{t}+p^{i T}_{t+1}A_{t}
       	\end{split} 
       	\end{equation}
       	Also,
       	\begin{equation}
       	\begin{split}
       	0 & = \frac{\partial H^{i}_{t}}{\partial u^{i \star}_{t}} \\ &=\left(B^{i}_{t}u^{i}_{t}\right)^{T}Q^{i}_{t+1}B^{i}_{t}+\left(x^{T}_{t}A^{T}_{t}+\left(\sum_{j \in n} B^{j}_{t}u^{j \star}_{t}\right)^{T}\right)Q^{i}_{t+1}B^{i}_{t}+u^{i T}_{t}R^{i}_{t}+p^{i T}_{t+1}B^{i}_{t} \\ &
       	= u^{i T}_{t}B^{i T}_{t}Q^{i}_{t+1}B^{i}_{t}+\left(x^{T}_{t}A^{T}_{t}+\left(\sum_{j \in n} B^{j}_{t}u^{j \star}_{t}\right)^{T}\right)Q_{t+1}B^{i}_{t}+u^{i T}_{t}R_{t}+p^{i T}_{t+1}B^{i}_{t}
       	\end{split} 
       	\end{equation}
       	Since $p^{i T}_{t+1}=0$ we have,
       	\begin{equation}\label{control_to_be_replaced}
       	\begin{split}
       	& u^{i T}_{t}\left(B^{i T}_{t}Q^{i}_{t+1}B^{i}_{t}+R^{i}_{t}\right)+x^{T}_{t}A^{T}_{t}Q^{i}_{t+1}B^{i}_{t}+\left(\sum_{j \in n} B^{j}_{t}u^{j \star}_{t}\right)^{T}Q^{i}_{t+1}B^{i}_{t}=0 
       	\end{split}
       	\end{equation}
       	Taking the transpose of the above equation we have,
       	\begin{equation}
       	\begin{split}
       	B^{iT}_{t}Q^{iT}_{t+1}B^{i}_{t}u^{i}_{t}+B^{iT}_{t}Q^{iT}_{t+1}\left[A_{t}x_{t}+\left(\sum_{j \in n}B^{j}_{t}u^{j \star}_{t}\right)\right]+R^{iT}_{t}u^{i}_{t}=0
       	\end{split}
       	\end{equation}
       	Simplifying and solving for $u^{i}_{t}$ we have
       	\begin{equation}
       	\begin{split}
       	u^{i}_{t}=-\left[R^{i T}_{t}+B^{i T}_{t}Q^{i}_{t+1}B^{i}_{t}\right]^{-1}B^{i T}_{t}\left[Q^{i}_{t+1}A_{t}x_{t}+Q^{i}_{t+1}\left(\sum_{j \in n} B^{j}_{t}u^{j \star}_{t}\right)\right]
       	\end{split} 
       	\end{equation}
       	This gives the optimal control for the player $i$. So we write it as
       	\begin{equation}
       	\begin{split}
       	& u^{i \star}_{t}=-\left[R^{i T}_{t}+B^{i T}_{t}Q^{i}_{t+1}B^{i}_{t}\right]^{-1}B^{i T}_{t}\left[Q^{i}_{t+1}A_{t}x^{\star}_{t}+Q^{i}_{t+1}\left(\sum_{j \in n} B^{j}_{t}u^{j \star}_{t}\right)\right] \\ &
       	\text{Replacing the value of $P^{i}_{t}=\left[R^{i T}_{t}+B^{i T}_{t}Q^{i}_{t+1}B^{i}_{t}\right]^{-1}B^{i T}_{t}$, $Q^{i}_{t+1}=S^{i}_{t+1}$, $s^{i}_{t+1}=0$ from} \\ & \text{the definitions we have,} \\ &
       	u^{i \star}_{t}=-P^{i}_{t}\left[S^{i}_{t+1}A_{t}x^{\star}_{t}+S^{i}_{t+1}\left(\sum_{j \in n} B^{j}_{t}u^{j \star}_{t}\right)+s^{i}_{t+1}\right] \\ &
       	u^{i \star}_{t}=-P^{i}_{t}S^{i}_{t+1}A_{t}x^{\star}_{t}-P^{i}_{t}\left(s^{i}_{t+1}+S^{i}_{t+1}\left(\sum_{j \in n} B^{j}_{t}u^{j \star}_{t}\right)\right)
       	\\ &
       	\text{Now we go ahead to find the minimum cost for player $i$. We have,}
       	\\ &
       	 J(u^{i \star}) =\frac{1}{2}\sum_{t=1}^{T}(x^{T \star}_{t+1}Q^{i}_{t+1}x^{\star}_{t+1}+u^{i T \star}_{t}R^{i}_{t}u^{i \star}_{t}) \\ &
       	\text{Replacing the value of $x^{\star}_{t+1}$ from the state updation law we have the above expression as,} \\ &
       	=\frac{1}{2}\sum_{t=1}^{T}[\left(x^{T \star}_{t}A^{T}_{t}+u^{i T \star}_{t}B^{T}_{t}+\left(\sum_{j \in n} B^{j}_{t}u^{j \star}_{t}\right)^{T}\right)Q^{i}_{t+1}\left(A_{t}x^{\star}_{t}+B^{i}_{t}u^{i \star}_{t}+\left(\sum_{j \in n} B^{j}_{t}u^{j \star}_{t}\right)\right) \\ & +u^{i T \star}_{t}R^{i}_{t}u^{i \star}_{t}] \\ &
       	=\frac{1}{2}\sum_{t=1}^{T}[x^{T \star}_{t}A^{T}_{t}Q_{t+1}^{i}A_{t}x^{\star}_{t}+2x^{T \star}_{t}A^{T}_{t}Q^{i}_{t+1}\left(B^{i}_{t}u^{i \star}_{t}+\left(\sum_{j \in n} B^{j}_{t}u^{j \star}_{t}\right)\right) \\ & +\left(u^{i T \star}_{t}B^{i T}_{t}+\left(\sum_{j \in n} B^{j}_{t}u^{j \star}_{t}\right)^{T}\right)Q^{i}_{t+1}\left(B^{i}_{t}u^{i \star}_{t}+\left(\sum_{j \in n} B^{j}_{t}u^{j \star}_{t}\right)\right)+u^{i T \star}_{t}R^{i}_{t}u^{i \star}_{t}]
       	\end{split} 
       	\end{equation}
       	Puting the value of the solved $u^{i \star}_{t}$ in the avove equation we have $J(u^{i \star})$ as
       	\begin{equation}
       	\begin{split}
       	& = \frac{1}{2}\sum_{t=1}^{T}\left[x^{T \star}A^{T}_{t}Q^{i}_{t+1}A_{t}x^{\star}_{t}\right] \\ & +\sum_{t=1}^{T}x^{T \star}_{t}A^{T}_{t}Q^{i}_{t+1}\left(-B^{i}_{t}P^{i}_{t}Q^{i}_{t+1}A_{t}x^{\star}_{t}-B^{i}_{t}P^{i}_{t}Q^{i}_{t+1}\left(\sum_{j \in n} B^{j}_{t}u^{j \star}_{t}\right) + \left(\sum_{j \in n} B^{j}_{t}u^{j \star}_{t}\right)\right) \\ &
       	+\frac{1}{2}\sum_{t=1}^{T}\left(x^{\star T}_{t}A^{T}_{t}Q^{i}_{t+1}P^{i T}_{t}+\left(\sum_{j \in n} B^{j}_{t}u^{j \star}_{t}\right)^{T}Q_{t+1}P^{i T}_{t+1}\right)\left(B^{i T}_{t}Q^{i}_{t+1}B^{i}_{t}+R^{i}_{t}\right) \\ & \left(P^{i}_{t}Q^{i}_{t+1}A_{t}x^{\star}_{t}+P^{i}_{t}Q^{i}_{t+1}\left(\sum_{j \in n} B^{j}_{t}u^{j \star}_{t}\right)\right) \\ & -\sum_{t=1}^{T}\left(x^{\star T}_{t}A^{T}_{t}Q^{i}_{t+1}P^{i T}_{t}+\left(\sum_{j \in n} B^{j}_{t}u^{j \star}_{t}\right)^{T}Q^{i}_{t+1}P^{i T}_{t}\right)B^{i T}_{t}Q^{i}_{t+1}\left(\sum_{j \in n} B^{j}_{t}u^{j \star}_{t}\right)   \\ & + \frac{1}{2}\sum_{t=1}^{T}\left(\sum_{j \in n} B^{j}_{t}u^{j \star}_{t}\right)^{T}Q^{i}_{t+1}\left(\sum_{j \in n} B^{j}_{t}u^{j \star}_{t}\right)
       	\end{split} 
       	\end{equation}
       	which on simplifying produces the result in the statement of the problem.
       \end{proof}
       Coming back to the discussion of Discrete Time $H^{\infty}$ control and matrix Riccati equations we attempt to provide a solution of equations (\ref{Riccati_S}) and (\ref{Riccati_Sigma}). Thereafter we try to find the values of $\gamma$ satisfying these two equations, this problem is dealt numerically. 
       \section{Solution of matrix Riccati equations} 
       For the sake of convenience we again present the neccessary matrix equations to solve equations (\ref{Riccati_S}) and (\ref{Riccati_Sigma}) in this section. As we deal with Infinite Horizon case the stationary versions of all the equations in the previous subsection are obtained
       by removing the index $t$ or $t+1$ to all matrix-valued symbols. Rewriting the necessary matrix equations in time independent form we have,
       \begin{equation}
       \begin{split}
       \bar{A} = A-BR^{-1}P^{T} \; , \; \tilde{A}=A-L^{T}N^{-1}C
       \end{split} 
       \end{equation}
       where $R=G^{T}G$, $P=H^{T}G$, $L=ED^{T}$, $N=EE^{T}$. \\
       Along with the additional notations
       \begin{equation}
       \begin{split}
       & \Gamma=(S^{-1}+BR^{-1}B^{T}-\gamma^{-2}M)^{-1} \\ &
       \bar{S}=\bar{A}^{T}(S-\gamma^{-2}M)^{-1}\bar{A}+Q-PR^{-1}P^{T} \\ &
       \Delta=(\Sigma^{-1}+C^{T}NC-\gamma^{-2}Q)^{-1} \\ &
       \tilde{\Sigma}=\tilde{A}(\Sigma^{-1}-\gamma^{-2}Q)^{-1}\tilde{A}^{T}+M-L^{T}N^{-1}L \\ &
       \end{split} 
       \end{equation}
       Then the two matrix Riccati equations needed for our purpose can be written as 
       \begin{equation}\label{Riccati_S_stationary}
       \begin{split}
       S=\bar{A}^{T}\Gamma\bar{A}+Q-PR^{-1}P^{T}
       \end{split}
       \end{equation}
       \begin{equation}\label{Riccati_Sigma_stationary}
       \begin{split}
       \Sigma=\tilde{A}\Delta\tilde{A}^{T}+M-L^{T}N^{-1}L
       \end{split} 
       \end{equation}
       where $Q=H^{T}H$, $M=DD^{T}$ \\
       A tedious calculation leads us to the solution of equations (\ref{Riccati_S_stationary}) and (\ref{Riccati_Sigma_stationary}) given by
       \begin{equation}\label{Inverse_S_solution}
       \begin{split}
       [A-BG^{-1}H]S^{-1}[A-BG^{-1}H]^{T}=S^{-1}+(BG^{-1})(BG^{-1})^{T}-\gamma^{-2}DD^{T}
       \end{split} 
       \end{equation} 
       \begin{equation}\label{Inverse_Sigma_solution}
       \begin{split}
       [A-DE^{-1}C]^{T}\Sigma^{-1}[A-DE^{-1}C]=\Sigma^{-1}+(E^{-1}C)^{T}(E^{-1}C)-\gamma^{-2}H^{T}H
       \end{split} 
       \end{equation}
       As the above two equations are not easily solvable symbolically we tame the problem numerically by assuming numerical values for the matrices. For our purpose we take the matrices as
       \begin{equation}
       \begin{split}
       A=\begin{bmatrix} 1 & 0 \\ 0 & 1 \end{bmatrix}, \; B=C=D=E=G=H=\begin{bmatrix} 0 & 1 \\ 1 & 0 \end{bmatrix}
       \end{split} 
       \end{equation}
       Puting the above values in equations (\ref{Inverse_S_solution}) and (\ref{Inverse_Sigma_solution}) we have,
       \begin{equation}
       \begin{split}
       & \begin{bmatrix} 1 & -1 \\ -1 & 1 \end{bmatrix}S^{-1}{\begin{bmatrix} 1 & -1 \\ -1 & 1 \end{bmatrix}}^{T}=S^{-1}+\begin{bmatrix} 1 & 0 \\ 0 & 1 \end{bmatrix}[1-\gamma^{-2}] 
       \end{split} 
       \end{equation}
       Let $S^{-1}=\begin{bmatrix} a & b \\ c & d \end{bmatrix}$. Solving the above equation we get
       \begin{equation*}
       \begin{split}
       a=\frac{1}{3}[\gamma^{-2}-1] , b=\frac{2}{3}[\gamma^{-2}-1] , c= \frac{2}{3}[\gamma^{-2}-1] , d=\frac{1}{3}[\gamma^{-2}-1]
       \end{split}
       \end{equation*}
       Therefore $S^{-1}= \frac{\gamma^{-2}-1}{3} \begin{bmatrix} 1 & 2 \\ 2 & 1 \end{bmatrix}$ and $S=\frac{3}{\gamma^{-2}-1} \begin{bmatrix} -1 & 2 \\ 2 & -1 \end{bmatrix}$.
       Similarly solution for $\Sigma$ gives $\Sigma^{-1}= \frac{\gamma^{-2}-1}{3} \begin{bmatrix} 1 & 2 \\ 2 & 1 \end{bmatrix}$ and $\Sigma=\frac{3}{\gamma^{-2}-1} \begin{bmatrix} -1 & 2 \\ 2 & -1 \end{bmatrix}$. \\
       To arrive at a range for $\gamma$ we have the following conditions
       \begin{equation}
       \begin{split}
       & \rho(MS) < \gamma^{2} \; \text{and} \; \rho(\tilde{\Sigma}S) < \gamma^{2} \\ &
       \text{or} \\ &
       \rho(\Sigma Q) < \gamma^{2} \; \text{and} \; \rho(\Sigma\bar{S}) < \gamma^{2}
       \end{split} 
       \end{equation}
       where '$\rho$' stands for spectral radius or the largest eigenvalue of the matrices in the argument.
       \subsection{Calculation of Spectral Radius}
       \begin{equation}
       \begin{split}
       & \rho(MS)=\rho\left(\frac{1}{\gamma^{-2}-1}\begin{bmatrix} -1 & 2 \\ 2 & -1 \end{bmatrix} \right)=\frac{3}{1-\gamma^{-2}} \\ &
       \rho(\tilde{\Sigma}S)=\rho\left(\frac{18}{(1-\gamma^{-2})(1-4\gamma^{-2})} \begin{bmatrix} 1 & -1 \\ -1 & 1 \end{bmatrix}\right)=\frac{18}{(1-4\gamma^{-2})(1-\gamma^{-2})} \\ &
       \rho(\Sigma Q)= \rho\left(\frac{1}{\gamma^{-2}-1}\begin{bmatrix} -1 & 2 \\ 2 & -1 \end{bmatrix}\right)=\frac{3}{1-\gamma^{-2}} \\ &
       \rho(\Sigma \bar{S})=\rho\left(\frac{6(\gamma^{-4}-\gamma^{-2}-1)}{(1-\gamma^{-2}+\gamma^{-4})^{2}-4}\begin{bmatrix} 1 & -1 \\ -1 & 1 \end{bmatrix}\right)=\frac{6(\gamma^{-4}-\gamma^{-2}-1)}{(1-\gamma^{-2}+\gamma^{-4})^{2}-4}
       \end{split} 
       \end{equation}
       For  $\gamma$ to satisfy the Riccati equations we have,
       \begin{equation}\label{Gamma_values_1}
       \begin{split}
       & \rho(MS) < \gamma^{2} \; \text{or} \; \frac{3}{1-\gamma^{-2}} < \gamma^{2} \\ & \text{and} \\ &
       \rho(\tilde{\Sigma}S) < \gamma^{2} \; \text{or} \; \frac{18}{(1-4\gamma^{-2})(1-\gamma^{-2})}< \gamma^{2}
       \end{split}
       \end{equation}
       OR
       \begin{equation}\label{Gamma_values_2}
       \begin{split}
      & \rho(\Sigma Q) < \gamma^{2} \; \text{or} \; \frac{3}{1-\gamma^{-2}} < \gamma^{2} \\ & \text{and} \\ & \rho(\Sigma \bar{S}) < \gamma^{2} \; \text{or} \; \frac{6(\gamma^{-4}-\gamma^{-2}-1)}{(1-\gamma^{-2}+\gamma^{-4})^{2}-4} < \gamma^{2}
       \end{split} 
       \end{equation}
       From the first equation of equation (\ref{Gamma_values_1}) for $\gamma > 0$ we have  $\gamma > 2$ and from second equation of equation (\ref{Gamma_values_1}) for $\gamma > 0$ we have $\gamma > 4.78$ or $ \gamma > .42$. Combining the above conditions we have $\gamma > 4.78$  or $\gamma > 2$\\
       From the first equation of equation (\ref{Gamma_values_2}) for $\gamma > 0$ we have  $\gamma > 2$ and from second equation of equation (\ref{Gamma_values_2}) for $\gamma > 0$ we have $\gamma < 1.48$. Combining we do not get any range for $\gamma$.  
       \section{Conclusion and problems in future}
        We try to conclude the article with a brief discussion and open problems. In our work we have revealed a new face for differential calculus where it is shown the well known properties of conventional matrix product could still be applied while doing the differential calculus for Semi-Tensor product(STP) of matrices. While the rich ideas of Boolean Calculus\cite{cheng2011matrix} are still present we have shown that(STP) formalism can also be a new addition in the existing ideas of matrix differential calculus and it can be applied to Boolean networks. While the ideas of STP formalism are still developing our technique can be a new addition in the existing formalism. This may open up new research avenues in matrix differential calculus. The idea of implementing the Markov Decision Process theory in Boolean Networks is an absolutely new addition. While deriving the optimal control for the Boolean networks we have shown that at optimal condition the system falls into a cycle where the optimal control values and the state trajectory values are repeated at each discrete time. This is a new finding on which further research can be carried out. In our second part of the discussion that is \textit{Game Based Control} we have derived the optimal control and the optimal cost in a thought problem through Pontryagin Minimization Principle. While dealing with $H^{\infty}$ optimal control for discrete time systems for infinite horizon problem involving the time independent matrices, we showed an absolutely new technique for solving the equations by considering delta variations of the state equation and eventually replacing the delta variations by differentials which made the equations much easier to handle. Also the matrix Riccati equations that came up were solved successfully by considering special values of the matrices. Finally we derive a range for the \textit{attenuation level} $\gamma$ and draw an ending line on the tedious calculations.

	\bibliographystyle{unsrt}
	\bibliography{library}

\begin{thebibliography}{10}

\bibitem{cheng2014semi}
Daizhan Cheng, Hongsheng Qi, Fehuang He, Tingting Xu, and Hairong Dong.
\newblock Semi-tensor product approach to networked evolutionary games.
\newblock {\em Control Theory and Technology}, 12(2):198--214, 2014.

\bibitem{cheng2015modeling}
Daizhan Cheng, Fenghua He, Hongsheng Qi, and Tingting Xu.
\newblock Modeling, analysis and control of networked evolutionary games.
\newblock {\em IEEE Transactions on Automatic Control}, 60(9):2402--2415, 2015.

\bibitem{cheng2010analysis}
Daizhan Cheng, Hongsheng Qi, and Zhiqiang Li.
\newblock {\em Analysis and control of Boolean networks: a semi-tensor product
  approach}.
\newblock Springer Science \& Business Media, 2010.

\bibitem{cheng2010linear}
Daizhan Cheng and Hongsheng Qi.
\newblock A linear representation of dynamics of boolean networks.
\newblock {\em IEEE Transactions on Automatic Control}, 55(10):2251--2258,
  2010.

\bibitem{li2011structure}
Z~Li, Yin Zhao, and D~Cheng.
\newblock Structure of higher order boolean networks.
\newblock {\em J. Grad. School, the Chinese Acad. Sci.}, 2011.

\bibitem{cheng2009input}
Daizhan Cheng.
\newblock Input-state approach to boolean networks.
\newblock {\em IEEE Transactions on Neural Networks}, 20(3):512--521, 2009.

\bibitem{cheng2011stability}
Daizhan Cheng, Hongsheng Qi, Zhiqiang Li, and Jiang~B Liu.
\newblock Stability and stabilization of boolean networks.
\newblock {\em International Journal of Robust and Nonlinear Control},
  21(2):134--156, 2011.

\bibitem{zhang2012multi}
Zhenning Zhang and Daizhan Cheng.
\newblock Multi-agent competitive control systems.
\newblock In {\em Proceedings of the 10th World Congress on Intelligent Control
  and Automation}, pages 2263--2267. IEEE, 2012.

\bibitem{li2010algebraic}
Zhiqiang Li and Daizhan Cheng.
\newblock Algebraic approach to dynamics of multivalued networks.
\newblock {\em International Journal of Bifurcation and Chaos},
  20(03):561--582, 2010.

\bibitem{cheng2005semi}
Daizhan Cheng.
\newblock Semi-tensor product of matrices and its applications to dynamic
  systems.
\newblock In {\em New Directions and Applications in Control Theory}, pages
  61--79. Springer, 2005.

\bibitem{qi2014networked}
Hongsheng Qi, Daizhan Cheng, and Hairong Dong.
\newblock On networked evolutionary games part 1: Formulation.
\newblock {\em IFAC Proceedings Volumes}, 47(3):275--280, 2014.

\bibitem{cheng2014networked}
Daizhan Cheng, Fehuang He, and Tingting Xu.
\newblock On networked evolutionary games part 2: Dynamics and control.
\newblock {\em IFAC Proceedings Volumes}, 47(3):281--286, 2014.

\bibitem{cheng2003semi}
Daizhan Cheng, Yali Dong, et~al.
\newblock Semi-tensor product of matrices and its some applications to physics.
\newblock {\em Methods and Applications of Analysis}, 10(4):565--588, 2003.

\bibitem{cheng2010strategy}
Daizhan Cheng, Yin Zhao, and Yifen Mu.
\newblock Strategy optimization with its application to dynamic games.
\newblock In {\em 49th IEEE Conference on Decision and Control (CDC)}, pages
  5822--5827. IEEE, 2010.

\bibitem{limebeer1992game}
David~JN Limebeer, Brian~DO Anderson, Pramod~P Khargonekar, and Michael Green.
\newblock A game theoretic approach to h infinity control for time-varying
  systems.
\newblock {\em SIAM Journal on Control and Optimization}, 30(2):262--283, 1992.

\bibitem{bernhard1991lecture}
Pierre Bernhard.
\newblock A lecture on the game theoretic approach to h infinity optimal
  control.
\newblock {\em Preprint}, 1991.

\bibitem{basar1991dynamic}
Tamer Basar.
\newblock A dynamic games approach to controller design: disturbance rejection
  in discrete-time.
\newblock {\em IEEE transactions on automatic control}, 36(8):936--952, 1991.

\bibitem{bacsar1989disturbance}
Tamer Ba{\c{s}}ar.
\newblock Disturbance attenuation in lti plants with finite horizon: Optimality
  of nonlinear controllers.
\newblock {\em Systems \& control letters}, 13(3):183--191, 1989.

\bibitem{bacsar1989differential}
Tamer Ba{\c{s}}ar and Pierre Bernhard.
\newblock {\em Differential games and applications}.
\newblock Springer, 1989.

\bibitem{haurie2005dynamic}
Alain Haurie and Georges Zaccour.
\newblock {\em Dynamic games: theory and applications}, volume~10.
\newblock Springer Science \& Business Media, 2005.

\bibitem{bacsar1984theory}
T~Ba{\c{s}}ar.
\newblock Theory of dynamic games and its applications in large scale systems
  design and optimization.
\newblock {\em IFAC Proceedings Volumes}, 17(2):1127--1131, 1984.

\bibitem{francis1987course}
Bruce~A Francis.
\newblock {\em A course in H [infinity] control theory}.
\newblock Berlin; New York: Springer-Verlag, 1987.

\bibitem{doyle1989state}
John~C Doyle, Keith Glover, Pramod~P Khargonekar, and Bruce~A Francis.
\newblock State-space solutions to standard h/sub 2/and h/sub infinity/control
  problems.
\newblock {\em IEEE Transactions on Automatic control}, 34(8):831--847, 1989.

\bibitem{bacsar2008h}
Tamer Ba{\c{s}}ar and Pierre Bernhard.
\newblock {\em H-infinity optimal control and related minimax design problems:
  a dynamic game approach}.
\newblock Springer Science \& Business Media, 2008.

\bibitem{vrabie2013optimal}
Draguna Vrabie, Kyriakos~G Vamvoudakis, and Frank~L Lewis.
\newblock {\em Optimal adaptive control and differential games by reinforcement
  learning principles}, volume~2.
\newblock IET, 2013.

\bibitem{zhao2010optimal}
Yin Zhao, Zhiqiang Li, and Daizhan Cheng.
\newblock Optimal control of logical control networks.
\newblock {\em IEEE Transactions on Automatic Control}, 56(8):1766--1776, 2010.

\bibitem{pan1995h}
Zigang Pan and Tamer Ba{\c{s}}ar.
\newblock H infinity control of markovian jump systems and solutions to
  associated piecewise-deterministic differential games.
\newblock In {\em New trends in dynamic games and applications}, pages 61--94.
  Springer, 1995.

\bibitem{abouheaf2014multi}
Mohammed~I Abouheaf, Frank~L Lewis, Kyriakos~G Vamvoudakis, Sofie Haesaert, and
  Robert Babuska.
\newblock Multi-agent discrete-time graphical games and reinforcement learning
  solutions.
\newblock {\em Automatica}, 50(12):3038--3053, 2014.

\bibitem{basar1999dynamic}
Tamer Basar and Geert~Jan Olsder.
\newblock {\em Dynamic noncooperative game theory}, volume~23.
\newblock Siam, 1999.

\bibitem{bernstein1989lqg}
Dennis~S Bernstein and Wassim~M Haddad.
\newblock Lqg control with an h/sup infinity/performance bound: a riccati
  equation approach.
\newblock {\em IEEE Transactions on Automatic Control}, 34(3):293--305, 1989.

\bibitem{bacsar2017riccati}
Tamer Ba{\c{s}}ar and Jun Moon.
\newblock Riccati equations in nash and stackelberg differential and dynamic
  games.
\newblock {\em IFAC-PapersOnLine}, 50(1):9547--9554, 2017.

\bibitem{isaacs1999differential}
Rufus Isaacs.
\newblock {\em Differential games: a mathematical theory with applications to
  warfare and pursuit, control and optimization}.
\newblock Courier Corporation, 1999.

\bibitem{basar1985dynamic}
Tamer Basar.
\newblock Dynamic games and incentives.
\newblock In {\em Systems and optimization}, pages 1--13. Springer, 1985.

\bibitem{basar2018handbook}
Tamer Basar and Georges Zaccour.
\newblock {\em Handbook of Dynamic Game Theory}.
\newblock Springer, 2018.

\bibitem{bernhard2016robust}
Pierre Bernhard.
\newblock Robust control and dynamic games.
\newblock {\em Handbook of Dynamic Game Theory}, pages 1--30, 2016.

\bibitem{cheng2011matrix}
Daizhan Cheng, Yin Zhao, and Xiangru Xu.
\newblock Matrix approach to boolean calculus.
\newblock In {\em 2011 50th IEEE Conference on Decision and Control and
  European Control Conference}, pages 6950--6955. IEEE, 2011.

\end{thebibliography}
\end{document}